\documentclass[preprint,11pt]{elsarticle}

\usepackage[T1]{fontenc}
\usepackage{lmodern}
\usepackage{microtype}
\usepackage{amsmath,amssymb,amsthm,mathtools}
\usepackage{enumitem}
\usepackage{xcolor}
\usepackage{tikz}
\usetikzlibrary{arrows.meta,decorations.pathreplacing}

\usepackage{subcaption}
\usetikzlibrary{matrix,positioning}
\usepackage[hmargin=2.5cm,vmargin=3cm]{geometry} 

\usepackage[colorlinks=true,linkcolor=blue!45!black,citecolor=blue!45!black,urlcolor=blue!45!black]{hyperref}

\allowdisplaybreaks
\setlist{leftmargin=*,itemsep=1pt,topsep=3pt}

\theoremstyle{plain}
\newtheorem{theorem}{Theorem}[section]
\newtheorem{lemma}[theorem]{Lemma}
\newtheorem{problem}[theorem]{Problem}
\newtheorem{proposition}[theorem]{Proposition}
\newtheorem{corollary}[theorem]{Corollary}
\theoremstyle{remark}

\newcommand{\Nzero}{\mathbb N_0}

\newcommand{\abs}[1]{\lvert#1\rvert}

\begin{document}

\begin{frontmatter}

\title{Solutions to Two Problems of S\'ark\"ozy and S\'os on Additive Representation Functions}

\author[sjtu]{Peiru Kuang}
\ead{peiru_k@sjtu.edu.cn}
\author[sjtu]{Yan Wang\corref{corresponding}\fnref{funding}}
\ead{yan.w@sjtu.edu.cn}
\cortext[corresponding]{Corresponding author.}
\fntext[funding]{Supported by the National Key R\&D Program of China under Grant No.~2022YFA1006400 and the National Natural Science Foundation of China under Grant No.~12571376.}
\address[sjtu]{School of Mathematical Sciences, Shanghai Jiao Tong University, Shanghai 200240, China}

\begin{abstract}
For a set $A\subseteq\mathbb{N}_0$, let $r_1(A,n)$ denote the number of solutions of the equation $a+a^{\prime}=n$ with $a,a^{\prime}\in A$, and let $r_2(A,n)$ denote the number of such solutions subject to $a\le a^{\prime}$. These functions are called additive representation functions (as first considered by Erd\H{o}s, S\'ark\"ozy and S\'os). In this paper,
we resolve two problems posed by S\'ark\"ozy and S\'os in 1997. 
\begin{itemize}
    \item First, if $A$ is infinite and $r_2(A,2m+1)\ge r_2(A,2m)$ for every sufficiently large $m$, then the complement of $A$ is finite. This gives a negative answer to Problem 3.1 in~\cite{SarkozySos1997}.
    \item Secondly, there exist an arithmetic function $f$ satisfying $f(n) \to \infty$, $f(n+1) \ge f(n)$ for $n > n_0$, and $f(n) = o\left(\frac{n}{(\log n)^2}\right)$, and a set $A$ such that
\(
|r_1(A,n) - f(n)| = o((f(n))^{1/2})
\)
holds on a sequence of integers $n$ whose density is $1$. This gives a positive answer to Problem 3.3 in~\cite{SarkozySos1997}.
\end{itemize}
\end{abstract}

\end{frontmatter}

\section{Introduction}
Let $\mathbb{N}_{0}$ and $\mathbb{N}$ denote the sets of nonnegative and positive integers, respectively. For a subset $A \subseteq \mathbb{N}_{0}$ and $n \in \mathbb{N}_{0}$, the numbers of solutions of the equations
$$a+a^{\prime}=n, \quad a,a^{\prime}\in A$$
and
$$a+a^{\prime}=n, \quad a,a^{\prime}\in A, \quad a\le a^{\prime}$$
are denoted by \(r_1(A,n)\) and \(r_2(A,n)\), respectively, and are
called the additive representation functions of \(A\). These functions are not independent; we always have $2r_2(A,n)-1\leq r_1(A,n)\leq 2r_2(A,n)$.

The study of additive representation functions is closely related to Sidon's work~\cite{Sidon1932} in harmonic analysis and has subsequently developed through a combination of analytic, combinatorial and probabilistic methods. A set \(A\) is called a \emph{Sidon set},
or a \(B_2[1]\)-set, if the sums \(a+a'\), with \(a,a'\in A\) and
\(a\leq a'\), are all distinct. Equivalently, \(r_2(A,n)\leq1\) for
every \(n\in\Nzero\). More generally, \(A\) is called a
\(B_2[g]\)-set if \(r_2(A,n)\leq g\) for every \(n\in\Nzero\).

A natural direction in the study of additive representation functions is to investigate their regularity. One of the most basic properties is boundedness, which gives rise to the extremal problem of determining how large \(A\) can be when \(r_2(A,n)\) is uniformly bounded. 
The systematic study of this problem was initiated by Erd\H{o}s and Tur\'an~\cite{ErdosTuran}. Classical constructions were obtained by Singer~\cite{Singer1938} and by Bose and Chowla~\cite{BoseChowla1962}, while dense infinite Sidon sequences were later constructed by Ajtai, Koml\'os and Szemer\'edi~\cite{AjtaiKomlosSzemeredi1981}, and by Ruzsa~\cite{Ruzsa1998}.
The theory for \(B_h[g]\)-sets was studied by Green~\cite{Green2001} and
further developed by Cilleruelo, Ruzsa and Vinuesa~\cite{CillerueloRuzsaVinuesa2010}; see also O'Bryant's survey~\cite{OBryant2004}.

Another natural regularity question is whether an additive
representation function can be eventually increasing. Erd\H{o}s, S\'ark\"ozy and S\'os studied the monotonicity of the functions \(r_i(A,n)\)~\cite{ErdosSarkozySosIV,ErdosSarkozySosV}; see also
Balasubramanian~\cite{Balasubramanian1987}.
Erd\H{o}s and Tur\'an \cite{ErdosTuran} proved in 1941 that if $A \subseteq \mathbb{N}$ is infinite, then $r_{1}(A,n)$ cannot be constant from a certain point on. Dirac \cite{Dirac1951} proved an analogous result for $r_{2}(A,n)$. Erd\H{o}s, S\'ark\"ozy and S\'os~\cite{ErdosSarkozySosIV} proved in 1985 that $r_1(A,n)$ can be monotone for $n > n_0$ only in the trivial case when $A$ contains all the positive integers from a certain point on. The corresponding problem for \(r_2(A,n)\) proved more difficult. In 1997, S\'ark\"ozy and S\'os~\cite{SarkozySos1997} asked the following.
\begin{problem}[Sárközy--Sós~\cite{SarkozySos1997}]\label{p1}
Does there exist an infinite set $A$ such that $\mathbb{N}\setminus A$ is infinite and $r_{2}(A,n)$ is increasing from a certain point on?
\end{problem}
Throughout this paper, the term \emph{increasing} means \emph{nondecreasing}; the same terminology is used in \cite{ErdosSarkozySosIV,SarkozySos1997}.
While related partial results were obtained in \cite{ChenSarkozySosTang2005, Stumpf2020}, Problem \ref{p1} has remained open for nearly three decades. In an interesting way, the two representation functions \(r_1(A,n)\), \(r_2(A,n)\) behave completely differently. In this paper, we show that the answer to Problem \ref{p1} is negative in a strong sense. %To show \(\mathbb N\setminus A\) is finite, it suffices to show \(\mathbb N_0\setminus A\) is finite.

\begin{theorem}\label{thm:monotonicity}
Let \(A\subseteq\mathbb N_0\) be infinite. If there exists
\(m_0\in\mathbb N_0\) such that
\(
r_2(A,2m+1)\ge r_2(A,2m)
\)
for every \(m\ge m_0\), then \(\mathbb N_0\setminus A\) is finite.
\end{theorem}
Balasubramanian~\cite{Balasubramanian1987} proved that eventual monotonicity of \(r_2(A,n)\) implies
\(|[1,N]\setminus A|=O(\log N)\). In fact, our Theorem~\ref{thm:monotonicity}
shows that \(|[1,N]\setminus A|=O(1)\), which strengthens the result of Balasubramanian~\cite{Balasubramanian1987}.

A related line of research concerns the inverse problem for representation functions: given an arithmetic function \(f\), one asks whether there exists a set \(A\) whose representation function is equal, or sufficiently close,
to \(f\). Ruzsa~\cite{Ruzsa1990} constructed an additive basis whose representation function has bounded mean square. Nathanson
\cite{Nathanson2004,Nathanson2005} developed a general theory of inverse problems for representation functions and showed that broad classes of functions can be realized as representation functions of additive bases for the integers.

There are, however, strong restrictions on how
regularly a representation function can approximate a prescribed function. The first fundamental result in this direction is the Erd\H{o}s--Fuchs theorem, which states that for every \(c>0\), there
is no set \(A\subseteq\mathbb N\) for which
\[
 \sum_{n\leq N} r_1(A,n)
 =cN+o\bigl(N^{1/4}(\log N)^{-1/2}\bigr).
\]
Montgomery and Vaughan~\cite{MontgomeryVaughan1990} later strengthened this result by removing the logarithmic factor. While the Erd\H{o}s--Fuchs theorem~\cite{ErdosFuchs} approximates the sum of additive representation functions, Erdős and Sárközy \cite{ErdosSarkozyI,ErdosSarkozyII} obtained a corresponding obstruction to pointwise approximation. They proved that if $f(n)\rightarrow+\infty$, $f(n+1)\ge f(n)$ for $n>n_{0}$ and $f(n)=o(\frac{n}{(\log n)^{2}})$, then 
$$\max_{n\le N}|r_{1}(A,n)-f(n)|=o((f(N))^{1/2})$$
cannot hold. However, as Sárközy and Sós pointed out in \cite{SarkozySos1997}, the situation of the problem may change completely if a zero-density set of sums can be neglected. This led to the following problem.
\begin{problem}[Sárközy--Sós~\cite{SarkozySos1997}]\label{p2}
Does there exist an arithmetic function $f$ satisfying $f(n) \to \infty$, $f(n+1) \ge f(n)$ for $n > n_0$, and $f(n) = o\left(\frac{n}{(\log n)^2}\right)$, and a set $A$ such that
\[
|r_1(A,n) - f(n)| = o((f(n))^{1/2})
\]
holds on a sequence of integers $n$ whose density is $1$?
\end{problem}
Representation functions outside density-zero
exceptional sets were subsequently studied by Fang~\cite{Fang2022}. 
For a set \(E\subseteq\mathbb N\) and a real number \(x\geq1\), let
\(
E(x)=|E\cap[1,x]|.
\)
A set \(S\subseteq\mathbb N\) has density \(\delta\) if
\(
\lim_{x\to\infty}\frac{|S\cap[1,x]|}{x}=\delta.
\)
Our second result answers this problem affirmatively in a strong sense.

\begin{theorem}\label{thm:density}
There exist an infinite set \(A\subseteq\mathbb N_0\) and an increasing arithmetic function \(f\) such that
\(f(n)\to\infty\) and \(f(n)=O(\log\log n)\). Let
\(
E=\{n\in\mathbb N:|r_1(A,n)-f(n)|\ne1\}.
\)
Then
\(
|E\cap[1,x]|
=O\bigl(x^{15/16}(\log\log x)^2\bigr)
\)
as \(x\to\infty\).
\end{theorem}

\subsection{Proof overview}

%This paper resolves two problems posed by S\'ark\"ozy and S\'os \cite{SarkozySos1997}. The first concerns the eventual monotonicity of \(r_2(A,n)\). The second asks whether \(r_1(A,n)\) can approximate a slowly growing increasing function outside a set of density zero. 

Now we sketch the proofs of Theorems \ref{thm:monotonicity} and \ref{thm:density}.

\medskip
%\noindent \textbf{Monotonicity of \(r_2(A,n)\).}
\noindent \textbf{Proof sketch of Theorem \ref{thm:monotonicity}.}
Let \(B=\mathbb N_0\setminus A\). We first use the assumption \(r_2(A,2m+1)\geq r_2(A,2m)\) to obtain a bound for the generating function of \(B\). This shows that \(B\) is locally sparse. If \(B\) were infinite, we could choose a rapidly increasing sequence in \(B\). The Ramsey theorem then gives an infinite subsequence whose pairwise sums have the same form. Iterating the resulting translation step gives nested infinite subsequences \(X_0\supseteq X_1\supseteq\cdots\) and increasing shifts \(0=t_0<t_1<\cdots\) such that \(X_k+t_i\subseteq B\) whenever \(0\leq i\leq k\); see Figure~\ref{fig:iterated-translates}. Choose a sufficiently large \(x\in X_k\).
Then place many distinct elements \(x+t_i\) in the single interval \((x,2x]\), a contradiction to the local sparsity of \(B\). Hence \(B\) is finite.
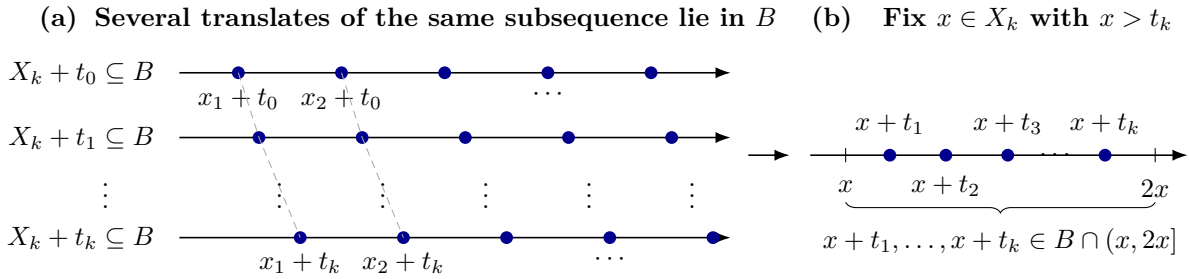
\begin{figure}[htbp]
\centering
\begin{tikzpicture}[
    x=0.78cm,
    y=0.78cm,
    >=Latex,
    axis/.style={->,semithick},
    point/.style={circle,fill=blue!55!black,inner sep=1.7pt},
    guide/.style={densely dashed,gray!65},
    every node/.style={font=\small}
]
% Left panel
\node[anchor=west,font=\small\bfseries] at (-2.55,4.05) {(a)};
\node[font=\small\bfseries] at (4.35,4.05)
    {Several translates of the same subsequence lie in \(B\)};

\foreach \y/\lab/\shift in {
    3.15/{X_k+t_0\subseteq B}/0,
    2.05/{X_k+t_1\subseteq B}/0.35,
    0.35/{X_k+t_k\subseteq B}/1.05
}{
    \draw[axis] (0,\y)--(9.35,\y);
    \node[anchor=east] at (-0.25,\y) {\(\lab\)};
    \foreach \z in {1.00,2.75,4.50,6.25,8.00}
        \node[point] at ({\z+\shift},\y) {};
}
\node at (-1.25,1.20) {\(\vdots\)};
\foreach \z in {1.70,3.45,5.20,6.95,8.70}
    \node at (\z,1.20) {\(\vdots\)};

\node[below=2pt] at (1.00,3.15) {\(x_1+t_0\)};
\node[below=2pt] at (2.75,3.15) {\(x_2+t_0\)};
\node[below=2pt] at (6.25,3.15) {\(\cdots\)};
\node[below=2pt] at (2.05,0.35) {\(x_1+t_k\)};
\node[below=2pt] at (3.80,0.35) {\(x_2+t_k\)};
\node[below=2pt] at (7.30,0.35) {\(\cdots\)};

\draw[guide] (1.00,3.15)--(1.35,2.05)--(2.05,0.35);
\draw[guide] (2.75,3.15)--(3.10,2.05)--(3.80,0.35);

% Arrow
\draw[->,semithick] (9.65,1.75)--(10.35,1.75);

% Right panel
\node[anchor=west,font=\small\bfseries] at (10.55,4.05)
    {(b)\quad Fix \(x\in X_k\) with \(x>t_k\)};
\draw[axis] (10.70,1.75)--(17.10,1.75);
\draw (11.30,1.63)--(11.30,1.87);
\draw (16.55,1.63)--(16.55,1.87);
\node[below=3pt] at (11.30,1.63) {\(x\)};
\node[below=3pt] at (16.55,1.63) {\(2x\)};

\node[point] at (12.05,1.75) {};
\node[point] at (13.00,1.75) {};
\node[point] at (14.05,1.75) {};
\node at (14.85,1.75) {\(\cdots\)};
\node[point] at (15.70,1.75) {};

\node[above=4pt] at (12.05,1.75) {\(x+t_1\)};
\node[below=4pt] at (13.00,1.75) {\(x+t_2\)};
\node[above=4pt] at (14.05,1.75) {\(x+t_3\)};
\node[above=4pt] at (15.70,1.75) {\(x+t_k\)};

\draw[decorate,decoration={brace,mirror,amplitude=4pt}]
    (11.30,0.95)--(16.55,0.95)
    node[midway,below=6pt,align=center]
    {\(x+t_1,\ldots,x+t_k\in B\cap(x,2x]\)};
\end{tikzpicture}
\caption{The iterative translation argument. At stage \(k\), the same infinite subsequence has \(k+1\) distinct translates contained in \(B\). For \(x\in X_k\) with \(x>t_k\), the points \(x+t_1,\ldots,x+t_k\) all lie in \(B\cap(x,2x]\). Taking \(k\) larger than the local sparsity bound gives the contradiction.}
\label{fig:iterated-translates}
\end{figure}

\medskip
\noindent \textbf{Proof sketch of Theorem \ref{thm:density}.}
%\noindent \textbf{Approximation of \(r_1(A,n)\).}
We use base-$4$ expansions. The idea is to
control the representations of \(n\) by splitting its digits in
several prescribed ways.

For each \(i\geq1\), we divide the digit positions into two
complementary periodic sets \(U_i^0\) and \(U_i^1\). We also choose
a threshold \(T_i\). The thresholds increase very rapidly. For each
\(\sigma\in\{0,1\}\), we include in \(A\) every integer at least
\(T_i\) whose nonzero digits occur only in positions from
\(U_i^\sigma\).

Now fix \(n\). For each \(i\), retain the digits of \(n\) in
\(U_i^0\) and set all other digits equal to zero. Do the same with \(U_i^1\). This gives two integers whose sum is \(n\). If both integers are at least \(T_i\), then both belong to \(A\). They give two ordered representations of \(n\). Figure~\ref{fig:digit-partitions} shows these partitions for \(i=1,2,3\).

We next show that almost every representation of \(n\) is obtained
in this way. Two difficulties may occur. First, the two elements in
a representation may come from digit sets that are not complementary.
Many digit positions are then missing from both sets. Secondly, we may have
\(n\geq T_i^2\), while one of the two parts defined by
\(U_i^0\) and \(U_i^1\) is smaller than \(T_i\). In that case, many
digits of \(n\) must be zero.

We place all integers arising from these two cases in an exceptional
set \(E_0\). We show that
\(
|E_0\cap[1,x]|
=
O\bigl(x^{15/16}(\log\log x)^2\bigr).
\)
It remains to count the representations when \(n\notin E_0\).
Let \(K(n)\) be the number of indices \(i\) for which
\(T_i^2\leq n\). Every such index gives two ordered representations.
There may be one more contributing index. Indeed, such an index must
satisfy
\(
2T_i\leq n<T_i^2.
\)
The thresholds are chosen so that the intervals
\([2T_i,T_i^2)\) are pairwise disjoint. Hence the number of indices that contribute representations is either \(K(n)\) or \(K(n)+1\). Since each such index gives exactly
two distinct ordered representations, and there are no other
representations outside \(E_0\), we have
$
r_1(A,n)\in\{2K(n),2K(n)+2\}
$.
Define
$
f(n)=2K(n)+1.
$
This is the integer between the two possible values of \(r_1(A,n)\).
Thus
$
|r_1(A,n)-f(n)|=1
$
for every \(n\notin E_0\). 

\begin{figure}[htbp]
\centering
\begin{tikzpicture}[
cell0/.style={
draw,
fill=black!15,
minimum width=0.48cm,
minimum height=0.48cm,
inner sep=0pt
},
cell1/.style={
draw,
fill=white,
minimum width=0.48cm,
minimum height=0.48cm,
inner sep=0pt
},
rowlabel/.style={
anchor=east,
font=\small
},
poslabel/.style={
font=\scriptsize
},
legend/.style={
font=\scriptsize,
anchor=west
}
]
\node[rowlabel] at (0,0.55) {digit position \(j\)};

\foreach \j in {0,...,15}{
\node[poslabel] at ({0.5+0.5*\j},0.55) {\j};
}

\foreach \i/\b/\y in {1/1/0,2/2/-0.58,3/4/-1.16}{
\node[rowlabel] at (0,\y) {\(i=\i\)};
\foreach \j in {0,...,15}{
\pgfmathtruncatemacro{\s}{mod(floor(\j/\b),2)}
\ifnum\s=0
\node[cell0] at ({0.5+0.5*\j},\y) {};
\else
\node[cell1] at ({0.5+0.5*\j},\y) {};
\fi
}
}

\node[cell0] at (2.2,-2.0) {};
\node[legend] at (2.5,-2.0) {positions in \(U_i^0\)};

\node[cell1] at (5.8,-2.0) {};
\node[legend] at (6.1,-2.0) {positions in \(U_i^1\)};
\end{tikzpicture}
\caption{The periodic partitions of the digit positions for \(i=1,2,3\).
Shaded cells belong to \(U_i^0\), and unshaded cells belong to \(U_i^1\).}
\label{fig:digit-partitions}
\end{figure}
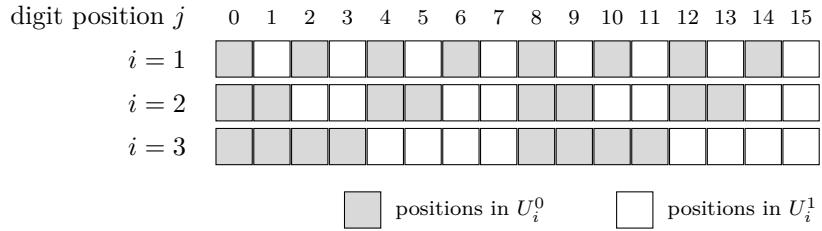

\medskip

\noindent \textbf{Notation.} 
Let $A_1 + A_2 + \dots + A_k = \{a_1 + a_2 + \dots + a_k : a_i \in A_i \text{ for each } i\}$. 
For functions \(F\) and \(G\), with \(G(x)>0\) for all sufficiently
large \(x\), the notation \(F(x)=O(G(x))\) means that
\(
|F(x)|\leq C G(x)
\)
for some constant \(C>0\) and all sufficiently large \(x\), while \(F(x)=o(G(x))\) means that \(F(x)/G(x)\to0\) as \(x\to\infty\).
A subscript in \(O_A(\,\cdot\,)\) indicates that the implied constant
may depend on \(A\). %All logarithms are natural unless a base is displayed explicitly.

\section{Proof of Theorem \ref{thm:monotonicity}}\label{sec:monotonicity}
Let \(F(z)=\sum_{a\in A}z^a\). For \(m\geq0\), let
\(d_m=r_2(A,2m+1)-r_2(A,2m)\).
Choose $m_0$ such that $d_m\geq0$ for every $m\geq m_0$, and define
\[
 \kappa=\sum_{m<m_0}\max\{-d_m,0\},
 \qquad
 D(q)=\sum_{m\geq0}d_mq^{2m}\quad(0<q<1).
\]
Then
\begin{equation}\label{mono:eq:Dlower}
 D(q)\geq-\kappa.
\end{equation}
The generating function of \(r_2(A,n)\) is
\[
 \sum_{n\geq0}r_2(A,n)z^n=\frac{F(z)^2+F(z^2)}2.
\]
Separating the even and odd coefficients, we obtain
\begin{equation}\label{mono:eq:D-F}
 4qD(q)
 =(1-q)F(q)^2-(1+q)F(-q)^2-2qF(q^2).
\end{equation}
Let $B=\Nzero\setminus A$ and
\[
 G(q)=\sum_{b\in B}q^b=\frac1{1-q}-F(q).
\]
We first obtain a logarithmic bound for the generating function of the complement.

\begin{lemma}\label{mono:lem:abelian-log}
We have
\[
 G(q)=O_{A}\!\left(\log\frac1{1-q}\right).
\]
\end{lemma}

\begin{proof}
Let $\alpha(q)=(1-q)F(q)$. By~\eqref{mono:eq:Dlower} and
\eqref{mono:eq:D-F},
\(
 (1-q)F(q)^2\geq2qF(q^2)-4q\kappa.
\)
Since $\alpha(q^2)=(1-q^2)F(q^2)$, we get
\(
 \alpha(q)^2
 \geq\frac{2q}{1+q}\alpha(q^2)-4q\kappa(1-q).
\)
Since $0\leq\alpha(q^2)\leq1$ and
$2q/(1+q)=1-(1-q)/(1+q)$, it follows that
\(
 \alpha(q)^2\geq\alpha(q^2)-(1+4\kappa)(1-q).
\)

Since $A$ is infinite, $F(u)\to\infty$ as $u$ tends to 1. Let \(s\in(0,1)\) be such that
\(
F(s^2)\geq 2(1+4\kappa).
\)
For \(q\in(s,1)\), let \(L\) be the least positive integer such that \(q^{2^L}<s\), and define
$q_i=q^{2^{L-i}}$, $0\leq i\leq L$.
Then \(q_L=q\) and \(q_{i-1}=q_i^2\). By the minimality of \(L\),
\(s^2\leq q_0<s\).
Moreover, since \(q_i=q_0^{2^{-i}}\),
\(
1-q_i\leq-\log q_i
=2^{-i}(-\log q_0)
\leq(-2\log s)2^{-i}.
\)

Let $\alpha_i=\alpha(q_i)$.
We have
\begin{equation}\label{e1}
\alpha_i^2
\geq
\alpha_{i-1}-(1+4\kappa)(1-q_i).    
\end{equation}
Furthermore,
\(
\alpha_{i-1}
= (1-q_{i-1})F(q_{i-1})
\geq (1-q_i)F(s^2)
\geq2(1+4\kappa)(1-q_i),
\)
since \(q_{i-1}\geq q_0\geq s^2\). Thus
\(\alpha_i\geq\sqrt{\alpha_{i-1}/2}\).
Also,
\(
\alpha_0=(1-q_0)F(q_0)
\geq(1-s)F(s^2)>0.
\)
Iterating \(\alpha_i\geq\sqrt{\alpha_{i-1}/2}\), we obtain
\(
\alpha_i\geq 2^{-1+2^{-i}}\alpha_0^{2^{-i}}.
\)
Since \(\alpha_0\) is bounded below by a positive constant depending only on
\(A\), there exists \(I=I(A)\) such that
\(\alpha_i\geq1/4\) for every \(i\geq I\).

Since \(\alpha_i\geq 1/4\), it follows that
\(
1-\alpha_i
\leq \frac{4}{5}(1-\alpha_{i-1})
+\frac{4(1+4\kappa)}{5}(1-q_i)
\) 
by \eqref{e1}.
Iterating this inequality yields
\[
1-\alpha_i
\leq
\left(\frac{4}{5}\right)^{i-I}(1-\alpha_I)
+\frac{4(1+4\kappa)}{5}
\sum_{j=I+1}^{i}
\left(\frac{4}{5}\right)^{i-j}(1-q_j).
\]
Since \(1-q_j=O_A(2^{-j})\), we have
\[
\begin{aligned}
\sum_{j=I+1}^{i}
\left(\frac45\right)^{i-j}(1-q_j)
&=
O_A\left(
\sum_{j=I+1}^{i}
\left(\frac45\right)^{i-j}2^{-j}
\right)\\
&=
O_A\left(
\left(\frac45\right)^i
\sum_{j=I+1}^{i}
\left(\frac58\right)^j
\right)\\
&=
O_A\left(\left(\frac45\right)^i\right).
\end{aligned}
\]
Hence \(1-\alpha_i=O_A((4/5)^i)\) and
\(
\sum_{i=I}^{L}(1-\alpha_i)=O_A(1).
\)

Since \(F(u)=1/(1-u)-G(u)\), we have
\(1-\alpha_i=(1-q_i)G(q_i)\). Substituting this identity into
\(
(1-\alpha_i)(1+\alpha_i)
\leq 1-\alpha_{i-1}+(1+4\kappa)(1-q_i),
\)
we obtain
\[
G(q_i)
\leq
\frac{1+q_i}{1+\alpha_i}G(q_{i-1})
+\frac{1+4\kappa}{1+\alpha_i}.
\]
For \(i\geq I\), the second term is \(O_A(1)\), while
\(
\frac{1+q_i}{1+\alpha_i}
\leq
\frac{2}{1+\alpha_i}
=
\frac{1}{1-(1-\alpha_i)/2}.
\)
Since \(0\leq(1-\alpha_i)/2\leq3/8\) and
$
-\log_2(1-x)\leq2x$
for $0\leq x\leq3/8$,
we have, for \(I\leq r\leq\ell\leq L\),
\[
\begin{aligned}
\prod_{i=r}^{\ell}
\frac{1+q_i}{1+\alpha_i}
\leq
\prod_{i=r}^{\ell}
\frac{1}{1-(1-\alpha_i)/2}
=
2^{
\sum_{i=r}^{\ell}
-\log_2\left(1-\frac{1-\alpha_i}{2}\right)
}
\leq
2^{\sum_{i=r}^{\ell}(1-\alpha_i)}
=O_A(1).
\end{aligned}
\]

If \(L<I\), then \(q<s^{2^{-I}}<1\), and hence
\(G(q)\leq 1/(1-s^{2^{-I}})=O_A(1)\). Suppose therefore that
\(L\geq I\). Since \(q_I=q_0^{2^{-I}}<s^{2^{-I}}<1\), we similarly have
\(G(q_I)=O_A(1)\). By iterating the preceding inequality and using the uniform bound for these
products, we obtain \(G(q)=G(q_L)=O_A(L+1)\).

By the minimality of \(L\), \(2^{L-1}(1-q)\leq 2^{L-1}(-\log q)\leq-\log s\). Therefore
\(
L\leq\log_2\frac{1}{1-q}+O_s(1).
\)
This proves the lemma.
\end{proof}

We next prove that the number of elements of \(B\) in \((x,2x]\)
is bounded independently of \(x\).

\begin{lemma}\label{mono:lem:dyadic}
There is an integer $M\geq1$ such that for every $x\geq1$,
\[
 |B\cap(x,2x]|\leq M.
\]
\end{lemma}

\begin{proof}
Substituting $F(q)=1/(1-q)-G(q)$ into~\eqref{mono:eq:D-F}, we obtain
\[
(1+q)\left(G(-q)-\frac1{1+q}\right)^2
   +2\bigl(G(q)-qG(q^2)\bigr) 
 =(1-q)G(q)^2+\frac1{1+q}-4qD(q).    
\]
Moreover,
\[
 G(q)-qG(q^2)
 =\sum_{b\in B}q^b(1-q^{b+1})\geq0.
\]
By Lemma~\ref{mono:lem:abelian-log} and~\eqref{mono:eq:Dlower}, $(1-q)G(q)^2+\frac1{1+q}-4qD(q)$ is bounded above as $q$ tends to 1. It follows that \(\sum_{b\in B}q^b(1-q^{b+1})\) is uniformly bounded. Let $X$ be sufficiently large and take $q=e^{-1/X}$. If $X<b\leq2X$,
then $q^b\geq e^{-2}$ and $1-q^{b+1}\geq1-e^{-1}$.
Thus every $b\in B\cap(X,2X]$ contributes at least
$e^{-2}(1-e^{-1})$ to this sum. This proves Lemma \ref{mono:lem:dyadic} for all sufficiently large \(X\). The values of \(X\) in a bounded interval are controlled by increasing \(M\), since \(B\cap[0,2X_0]\) is finite for every fixed \(X_0\).
\end{proof}

For $n\geq0$, let
\[
 B(n)=|B\cap[0,n]|,
 \qquad
 \rho_{B}(n)
 =|\{(b,b')\in B^2:b<b',\ b+b'=n\}|.
\]

\begin{lemma}\label{mono:lem:complement-identity}
For every $n\geq0$,
\begin{equation}\label{mono:eq:complement-identity}
 r_2(A,n)
 =\left\lfloor\frac n2\right\rfloor+1-B(n)+\rho_{B}(n).
\end{equation}
Thus, for all sufficiently large $m$,
\begin{equation}\label{mono:eq:rho-ineq}
 \rho_{B}(2m+1)\geq\rho_{B}(2m).
\end{equation}
\end{lemma}

\begin{proof}
There are \(\lfloor n/2\rfloor+1\) unordered representations of
\(n\) in \(\Nzero\). Counting the elements of \(B\cap[0,n]\) occurring
in these representations counts each representation with two distinct
entries in \(B\) twice and every other representation meeting \(B\)
once. Hence
\[
r_2(A,n)=\left\lfloor\frac n2\right\rfloor+1-B(n)+\rho_B(n).
\]
Since \(B(2m+1)-B(2m)=|B\cap\{2m+1\}|\), we obtain
\[
d_m=\rho_B(2m+1)-\rho_B(2m)-|B\cap\{2m+1\}|.
\]
Thus \(d_m\ge0\) implies
\(\rho_B(2m+1)\ge\rho_B(2m)\).
\end{proof}

It follows that if a sufficiently large even integer is the sum of two distinct elements of \(B\), then the following odd integer is also the sum of two distinct elements of \(B\).

\begin{corollary}\label{mono:cor:successor-sum}
If $x<y$ are in $B$, $x\equiv y\pmod2$, and $x+y$ is sufficiently large,
then there exist $u<v$ in $B$ such that \(u+v=x+y+1\).
\end{corollary}

\begin{proof}
Write \(x+y=2m\). Since \(x<y\) and \(x,y\in B\), we have
\(\rho_B(2m)\ge1\). Hence, for all sufficiently large \(m\),
\(
\rho_B(2m+1)\ge \rho_B(2m)\ge1,
\)
so there exist \(u<v\) in \(B\) with \(u+v=2m+1\).
\end{proof}
%We shall use the following standard form of the infinite Ramsey theorem~\cite{Ramsey1930}.
%\begin{theorem}[Ramsey~\cite{Ramsey1930}]\label{Ramsey}
%Let $\Gamma$ be an infinite class, and $\mu$ and $r$ positive integers; and let all those sub-classes of $\Gamma$ which have exactly $r$ members, or, as we may say, let all $r$-combinations of the members of $\Gamma$ be divided in any manner into $\mu$ mutually exclusive classes $C_i$ $(i=1,2,\ldots,\mu)$, so that every $r$-combination is a member of one and only one $C_i$; then, assuming the axiom of selections, $\Gamma$ must contain an infinite sub-class $\Delta$ such that all the $r$-combinations of the members of $\Delta$ belong to the same $C_i$.
%\end{theorem}

The next lemma finds two further translations contained in \(B\).

\begin{lemma}\label{mono:lem:two-translates}
Suppose $B\subseteq\Nzero$ satisfies $|B\cap(x,2x]|\le M$ for some constant $M$ and all $x\ge1$. 
Moreover, suppose that for any $x<y\in B$ of the same parity with $x+y$ sufficiently large, there exist $u<v\in B$ such that $u+v=x+y+1$. Let $X=\{x_1<x_2<\cdots\}$ be a sequence of integers of one parity such that
\begin{equation}\label{mono:eq:rapid-growth}
\frac{x_{j-1}}{x_j}\longrightarrow0.
\end{equation}
If $X+t\subseteq B$ for some integer $t\geq0$, then there exist an infinite subsequence $X'\subseteq X$ and integers $p,p'$ such that $p+p'=2t+1$, $X'+p\subseteq B$, and $X'+p'\subseteq B$. In particular, $\max\{p,p'\}\geq t+1$.
\end{lemma}

\begin{proof}
Color the
positive elements of \(B\) in increasing order using \(M+1\) colors. When an element 
\(y\) is colored,  we forbid colors of elements in \(B\cap[y/2,y)\), so there are at most \(M\) forbidden colors by assumption.
Thus a coloring can be chosen so that positive elements \(x<y\) of the
same color satisfy \(y>2x\).
If \(0\in B\), assign it a color not used for positive elements.

After deleting finitely many terms of \(X\), the sum
\((x_i+t)+(x_j+t)\) is sufficiently large whenever \(i<j\). By
Corollary~\ref{mono:cor:successor-sum}, for each \(i<j\), choose \(u_{ij}<v_{ij}\) in \(B\)
such that \(u_{ij}+v_{ij}=x_i+x_j+2t+1\). Color \(\{i,j\}\) by the ordered pair
of colors of \(u_{ij}\) and \(v_{ij}\), in this order. By Ramsey Theorem~\cite{Ramsey1930}, after passing to an infinite subsequence, we may assume that these colors are fixed, say \(U\) and \(V\).

Let \(C=2t+1\). By (\ref{mono:eq:rapid-growth}), delete finitely many further terms and relabel the
sequence so that, for every \(i<j\),
\(v_{ij}>(x_i+x_j+C)/2>x_j/2\) and
\(v_{ij}\leq x_i+x_j+C\le x_{j-1}+x_j+C<2x_j\).
Thus all \(v_{ij}\), \(i<j\), lie in \((x_j/2,2x_j)\). This interval contains at most two elements of color \(V\): three such elements \(y_1<y_2<y_3\) would satisfy \(y_3>2y_2>4y_1>2x_j\), a contradiction to \(y_3<2x_j\).

For each fixed \(j\), let
$
  \mathcal V_j:=\{v_{ij}:1\le i<j\}.
$
As shown above, \(\mathcal V_j\) contains at most two distinct values.
Let
$
  v_j:=\max \mathcal V_j.
$
If \(\mathcal V_j\setminus\{v_j\}\neq\varnothing\), denote its unique
element by \(w_j\). Thus \(w_j<v_j\). We claim that
\[
  \bigl|\{i<j:v_{ij}=w_j\}\bigr|\le 1.
\]
%Let \(v_j\) be the largest value among the \(v_{ij}\), \(i<j\).
%If there is another value, denote it by \(w_j<v_j\). We claim that
%\(v_{ij}=w_j\) for at most one \(i<j\).

Assume that \(v_{ij}=v_{kj}=w_j\) for some \(i<k<j\). Then \(u_{ij}\ne u_{kj}\). Since $0$ has a color not used for positive
elements, \(u_{ij}\) and \(u_{kj}\) are positive elements of color \(U\), so \(u_{kj}>2u_{ij}\). Since
\(u_{kj}-u_{ij}=x_k-x_i\), we have
\(u_{ij}<x_k-x_i\le x_{j-1}\), and hence
\(w_j=x_i+x_j+C-u_{ij}>x_j+C-x_{j-1}\). Also, \(v_j=v_{hj}\) for some \(h<j\), so
\(v_j\le x_h+x_j+C\le x_{j-1}+x_j+C\). Since \(w_j\) and \(v_j\)
have the same color, \(v_j>2w_j\). Therefore
\(x_j+x_{j-1}+C>2x_j+2C-2x_{j-1}\), and hence
\(x_j+C<3x_{j-1}\), a contradiction to
\eqref{mono:eq:rapid-growth}. Thus \(v_{ij}=v_j\) for all but at
most one \(i<j\).

Fix three indices \(i_1<i_2<i_3\). For every \(j>i_3\), at most one \(r\in\{1,2,3\}\) satisfies \(v_{i_rj}\ne v_j\). Hence at least one of the three pairs $\{i_1,i_2\}$, $\{i_1,i_3\}$ and $\{i_2,i_3\}$ has both of its indices \(i\) satisfying \(v_{ij}=v_j\). By Ramsey Theorem~\cite{Ramsey1930}, there are fixed indices \(a<b\) and an infinite set \(J\) such that
\(
v_{aj}=v_{bj}=v_j
\)
for every \(j\in J\).

For \(j\in J\), let \(c_j=x_j+C-v_j\). Then \(x_a+c_j\) and
\(x_b+c_j\) are distinct positive elements of color \(U\), so
\(x_b+c_j>2(x_a+c_j)\). Thus
\(-x_a<c_j<x_b-2x_a\). After passing to an infinite subset of
\(J\), we may assume that \(c_j=p\) for every \(j\in J\). Let \(p'=C-p\). Then \(x_j+p'=v_j\in B\) for every \(j\in J\).
For each \(j\in J\), at most one \(i<j\) satisfies
\(v_{ij}\ne v_j\). For each \(j\in J\), let
$
  E_j:=\{i<j:v_{ij}\ne v_j\}.
$
We have already shown that \(|E_j|\le 1\). Call an index \(i\in J\) bad if \(i\in E_j\) for every \(j\in J\) with \(j>i\). There is at most one bad index. Indeed, if \(i<k\) were two bad indices, then, choosing \(j\in J\) with \(j>k\), we would have \(i,k\in E_j\), a contradiction to \(|E_j|\le1\). Remove the bad index, if it exists.
Then, for every remaining \(i\in J\), there exists \(j\in J\) with
\(j>i\) and \(v_{ij}=v_j\).
Let \(X'\) be the remaining set
\(\{x_i:i\in J\}\).

For each \(x_i\in X'\), choose \(j\in J\) with \(j>i\) and
\(v_{ij}=v_j\). Then
\(u_{ij}=x_i+x_j+C-v_j=x_i+p\in B\). Also, since \(i\in J\), we have \(c_i=p\), and hence
\(x_i+p'=v_i\in B\). Therefore \(X'+p\subseteq B\) and
\(X'+p'\subseteq B\). Finally, \(p+p'=2t+1\), so
\(\max\{p,p'\}\ge t+1\).
\end{proof}

Iterating the preceding lemma contradicts the dyadic bound.
\begin{proof}[Proof of Theorem~\ref{thm:monotonicity}]
Assume for contradiction that $B=\Nzero\setminus A$ is
infinite. One of the two parity classes of $B$ is infinite.
Choose from it a sequence
\(
 X_0=\{x_1<x_2<\cdots\}
\)
such that $x_{j-1}/x_j\to0$. Let $t_0=0$. Then
$X_0+t_0\subseteq B$.

We inductively construct nested infinite subsequences
\(
 X_0\supseteq X_1\supseteq\cdots\supseteq X_{M+1}
\)
and integers
\(
 0=t_0<t_1<\cdots<t_{M+1}
\)
such that
\begin{equation}\label{mono:eq:all-translates}
 X_k+t_i\subseteq B
 \qquad(0\leq i\leq k).
\end{equation}
Suppose $X_k$ and $t_0,\ldots,t_k$ have been constructed. By Lemma~\ref{mono:lem:two-translates}, there exist an infinite subsequence $X_{k+1}\subseteq X_k$ and integers $p,p'$ such that $p+p'=2t_k+1$, $X_{k+1}+p\subseteq B$, and $X_{k+1}+p'\subseteq B$. Setting $t_{k+1}=\max\{p,p'\}$, we have $t_{k+1}\geq t_k+1$. Since $X_{k+1}\subseteq X_k$, all earlier containments in \eqref{mono:eq:all-translates} remain valid.

Choose $x\in X_{M+1}$ such that $x>t_{M+1}$. By
\eqref{mono:eq:all-translates},
\(
 x+t_1,\ldots,x+t_{M+1}\in B\cap(x,2x].
\)
These are $M+1$ distinct elements of \(B\cap(x,2x]\). By
Lemma~\ref{mono:lem:dyadic}, this set has at most \(M\) elements, a
contradiction. Therefore $B$ is finite.
\end{proof}

\section{Proof of Theorem \ref{thm:density}}\label{sec:density}

In this section, we prove Theorem~\ref{thm:density}. All digit expansions in this section are of base $4$. For every fixed 
\(x\in\mathbb N_0\), write
\[
 x=\sum_{j\ge0}d_j(x)4^j,
 \qquad d_j(x)\in\{0,1,2,3\},
\]
where all but finitely many digits vanish. For \(i\ge1\) and \(\sigma\in\{0,1\}\), let
\[
 U_i^\sigma=\left\{j\ge0:\left\lfloor\frac{j}{2^{i-1}}\right\rfloor\equiv\sigma\pmod 2\right\}.
\]
Thus \(U_i^0\) and \(U_i^1\) partition $\mathbb N_0$, and both sets are periodic with period \(2^i\). For \(X\subseteq\mathbb N_0\), define
\[
 \mathcal B(X)=\{x\in\mathbb N_0:d_j(x)=0\text{ for every }j\notin X\}.
\]
Let $L_i:=16\cdot4^i$ and $T_i:=4^{L_i}$.
We call the labels \((i,\sigma)\) and \((k,\tau)\) \emph{complementary} if \(i=k\) and \(\sigma=1-\tau\).
For an integer \(m\ge1\), define \([0,m):=\{0,1,\ldots,m-1\}\).
The following lemma shows that two noncomplementary sets \(U_i^\sigma\) and \(U_k^\tau\) leave a positive proportion of the indices uncovered.
\begin{lemma}\label{dens:lem:missing}
Suppose that \(L_i,L_k<m\), and that
\((i,\sigma)\) and \((k,\tau)\) are not complementary. Then
\[
  \abs{[0,m)\setminus\bigl(U_i^\sigma\cup U_k^\tau\bigr)}
  \ge \frac{m}{8}.
\]
\end{lemma}

\begin{proof}
Suppose first that \(i\ne k\). By symmetry, assume that \(i<k\). The set \(U_k^\tau\) has period \(2^k\) and is constant on each block of length \(2^{k-1}\). The set \(U_i^\sigma\) has period \(2^i\), and \(2^i\mid 2^{k-1}\). Hence, in each period of length \(2^k\), exactly one quarter of the positions belong to neither \(U_i^\sigma\) nor \(U_k^\tau\). Write \(m=q2^k+r\), where \(0\le r<2^k\). Each complete period contains exactly \(2^{k-2}\) positions belonging to neither \(U_i^\sigma\) nor \(U_k^\tau\). Therefore
\(
  \abs{[0,m)\setminus\bigl(U_i^\sigma\cup U_k^\tau\bigr)}
  \ge q2^{k-2}
  =\frac{m-r}{4}
  \ge\frac{m}{4}-2^{k-2}>\frac{m}{4}-2^{k}.
\)
Since \(L_k<m\), we have
\(
  m>16\cdot 4^k=16\cdot 2^{2k}\ge 32\cdot 2^k.
\)
Thus \(2^k<m/32\), and hence
\(
  \frac{m}{4}-2^k
  >
  \frac{m}{4}-\frac{m}{32}
  >
  \frac{m}{8}.
\)

Now suppose that \(i=k\). Since the labels are not complementary, we must have \(\sigma=\tau\). In each period of length \(2^i\), exactly half of the positions lie outside \(U_i^\sigma\). Thus,
\(
  \abs{[0,m)\setminus U_i^\sigma}
  \ge \frac{m}{2}-2^i>
  \frac{m}{2}-\frac{m}{32}
  >
  \frac{m}{8}.
\)
This completes the proof.
\end{proof}
The next lemma bounds the number of possible sums in terms of the number of uncovered digit indices.
\begin{lemma}\label{dens:lem:digit-sums}
Let \(X,Y\subseteq\mathbb N_0\), and suppose that at least \(g\) integers in \([0,m)\) belong to neither \(X\) nor \(Y\). Then
\[
 \abs{(\mathcal B(X)+\mathcal B(Y))\cap[0,4^m)}\le 4^m2^{-g}.
\]
\end{lemma}

\begin{proof}
Let \(n=a+b<4^m\), where \(a\in\mathcal B(X)\) and \(b\in\mathcal B(Y)\). Then \(a,b<4^m\). Since \(n=a+b\), there are integers \(c_j\geq0\), with \(c_0=0\), such that
\(d_j(a)+d_j(b)+c_j=d_j(n)+4c_{j+1}\) for every \(j\geq0\).
Since \(d_j(n)\in\{0,1,2,3\}\), it follows that
\[
  c_{j+1}
  =
  \left\lfloor
  \frac{d_j(a)+d_j(b)+c_j}{4}
  \right\rfloor.
\]
Since \(d_j(a)+d_j(b)+c_j\le7\), by induction, we have \(c_j\in\{0,1\}\) for every \(j\). If \(j\notin X\cup Y\), then \(d_j(a)=d_j(b)=0\). Hence
\(
  c_j=d_j(n)+4c_{j+1}.
\)
Since \(c_j\in\{0,1\}\), we must have \(c_{j+1}=0\) and
\(d_j(n)=c_j\in\{0,1\}\). Thus, at each of the \(g\) integers
outside \(X\cup Y\), the digit \(d_j(n)\) has at most two possible values. At each of the remaining \(m-g\) positions, it has at most four possible values. Since every integer \(n<4^m\) is uniquely determined by the digits \(d_0(n),\ldots,d_{m-1}(n)\), the number of possible sums is at most
\(
  2^g4^{m-g}=4^m2^{-g}.
\)
\end{proof}

Define the digit projection
\[
 \pi_i^\sigma(n)=\sum_{j\in U_i^\sigma}d_j(n)4^j.
\]
Since \(U_i^0\) and \(U_i^1\) partition the digit positions, we have
\(
  n=\pi_i^0(n)+\pi_i^1(n).
\)

\begin{lemma}\label{dens:lem:small-projection}
If \(2L_i<m\), then, for each \(\sigma\in\{0,1\}\),
\(
  \abs{\{0\le n<4^m:\pi_i^\sigma(n)<4^{L_i}\}}
  \le 4^{7m/8}.
\)
\end{lemma}

\begin{proof}
Since \(\pi_i^\sigma(n)=\sum_{j\in U_i^\sigma}d_j(n)4^j<4^{L_i}\) and all terms in the sum are nonnegative, we must have
\(d_j(n)=0\) for every \(j\in U_i^\sigma\cap[L_i,m)\). The set \(U_i^\sigma\) has period \(2^i\), and each complete period
contains exactly \(2^{i-1}\) elements of \(U_i^\sigma\). Remove
fewer than \(2^i\) positions from each end of \([L_i,m)\) so that
the remaining positions form complete periods. It follows that
\[
  \abs{U_i^\sigma\cap[L_i,m)}
  \ge \frac{m-L_i}{2}-2^i>\frac{m}{4}-2^i.
\]
Also,
\(m>2L_i=32\cdot4^i=32\cdot2^{2i}\ge64\cdot2^i\), so
\(2^i<m/64\). Hence
\[
  \abs{U_i^\sigma\cap[L_i,m)}
  >\frac m4-\frac m{64}>\frac m8.
\]

Hence at least \(m/8\) of the digits \(d_j(n)\) with $j\in U_i^\sigma\cap[L_i,m)$ are zero. The other digits have at most four choices each, so the number of possible integers \(n\) is at most
\(
4^{m-m/8}=4^{7m/8}.
\)
\end{proof}

For \(i\geq1\) and \(\sigma\in\{0,1\}\), let
\(
D_i^\sigma=\mathcal B(U_i^\sigma)\cap[T_i,\infty).
\)
Let \(E_{\mathrm{nc}}\) be the union of the sets
\(D_i^\sigma+D_k^\tau\) over all noncomplementary pairs
\((i,\sigma)\) and \((k,\tau)\).

\begin{proposition}\label{dens:prop:enc}
For \(m\ge1\),
\(
  \abs{E_{\mathrm{nc}}\cap[0,4^m)}
  =O\bigl((\log m)^2 4^m2^{-m/8}\bigr).
\)
\end{proposition}

\begin{proof}
Suppose that \(n=a+b<4^m\), where
\(a\in D_i^\sigma\) and \(b\in D_k^\tau\). Since
\(T_i\le a<n<4^m\) and \(T_k\le b<n<4^m\), we have
\(L_i,L_k<m\). Since \(L_i=16\cdot4^i\), there are only \(O(\log m)\) possible indices \(i\), and hence only \(O(\log m)\) relevant labels \((i,\sigma)\). Thus there are \(O((\log m)^2)\) relevant pairs of labels.

Fix a noncomplementary pair \((i,\sigma)\), \((k,\tau)\).
By Lemma~\ref{dens:lem:missing}, at least \(m/8\) integers in
\([0,m)\) belong to neither \(U_i^\sigma\) nor \(U_k^\tau\).
Since \(D_i^\sigma\subseteq\mathcal B(U_i^\sigma)\) and
\(D_k^\tau\subseteq\mathcal B(U_k^\tau)\), by
Lemma~\ref{dens:lem:digit-sums},
\(
  \abs{(D_i^\sigma+D_k^\tau)\cap[0,4^m)}
  \le 4^m2^{-m/8}.
\)
Summing this estimate over the relevant pairs proves the result.
\end{proof}

%Let \(E_{\mathrm{pr}}\) be the set of integers \(n\) for which \(T_i^2\le n\) and \(\min\{\pi_i^0(n),\pi_i^1(n)\}<T_i\) for some \(i\ge1\).
Let
\[
  E_{\mathrm{pr}}
  :=
  \left\{
    n\in\mathbb N_0:
    \text{for some }i\ge1,\ 
    T_i^2\le n
    \text{ and }
    \min\{\pi_i^0(n),\pi_i^1(n)\}<T_i
  \right\}.
\]

\begin{proposition}\label{dens:prop:epr}
For \(m\ge1\),
\(
  \abs{E_{\mathrm{pr}}\cap[0,4^m)}
  =O\bigl((\log m)4^{7m/8}\bigr).
\)
\end{proposition}

\begin{proof}
Let $n \in E_{pr}$.
Since \(T_i^2\le n<4^m\), we have
\(4^{2L_i}=T_i^2<4^m\), and hence \(2L_i<m\).
Since \(L_i=16\cdot4^i\), there are only \(O(\log m)\) relevant
labels \((i,\sigma)\). For each such label, by Lemma~\ref{dens:lem:small-projection},
at most \(4^{7m/8}\) integers \(n<4^m\) satisfying
\(\pi_i^\sigma(n)<T_i\). Summing over the relevant labels proves
Proposition \ref{dens:prop:epr}.
\end{proof}

Let \(E_0=E_{\mathrm{nc}}\cup E_{\mathrm{pr}}\), which will be the exceptional set in Theorem \ref{thm:density}. 

\begin{corollary}\label{dens:cor:exceptional}
For every \(x\ge3\),
\(
 \abs{E_0\cap[1,x]}=O\bigl(x^{15/16}(\log\log x)^2\bigr).
\)
In particular, \(E_0\) has density zero.
\end{corollary}

\begin{proof}
Choose \(m\) so that \(4^{m-1}\le x<4^m\). By Propositions~\ref{dens:prop:enc} and~\ref{dens:prop:epr}, we have
\[
 \abs{E_0\cap[1,x]}
 \le O\bigl((\log m)^2(4^m)^{15/16}
 +(\log m)(4^m)^{7/8}\bigr).
\]
Since \(4^m\le4x\) and \(\log m=O(\log\log x)\), this completes the proof.
\end{proof}

Define
\[
  A=\bigcup_{i\ge1}\bigl(D_i^0\cup D_i^1\bigr).
\]
For \(a\in A\), let
\(
 \Lambda(a)=\{(i,\sigma):a\in D_i^\sigma\}
\)
be its set of labels. This set is finite because \(T_i\to\infty\). The sets \(D_i^\sigma\) need not be disjoint, so we need the following observation.
Let $
  J(n)=\{i\ge1:\pi_i^0(n)\ge T_i
  \text{ and }\pi_i^1(n)\ge T_i\}.
$

\begin{lemma}\label{dens:lem:rigidity}
Let \(n\notin E_{\mathrm{nc}}\). Every ordered representation
\(n=a+b\), with \(a,b\in A\), is of the form
\(
(a,b)=\bigl(\pi_i^\sigma(n),\pi_i^{1-\sigma}(n)\bigr)
\)
for some \(i\in J(n)\) and \(\sigma\in\{0,1\}\). Moreover, the
ordered pairs
$\bigl(\pi_i^\sigma(n),\pi_i^{1-\sigma}(n)\bigr)$,
$i\in J(n)$, $\sigma\in\{0,1\}$,
are pairwise distinct. Thus,
\(
r_1(A,n)=2|J(n)|.
\)
\end{lemma}

\begin{proof}
Choose labels \((i,\sigma)\in\Lambda(a)\) and
\((k,\tau)\in\Lambda(b)\). Since \(n\notin E_{\mathrm{nc}}\), these
labels are complementary. Hence \(i=k\) and \(\tau=1-\sigma\). Since
the nonzero digits of \(a\) and \(b\) occur in the disjoint sets
\(U_i^\sigma\) and \(U_i^{1-\sigma}\), respectively, we have
\(a=\pi_i^\sigma(n)\) and \(b=\pi_i^{1-\sigma}(n)\). Moreover,
\(a,b\ge T_i\), so \(i\in J(n)\). Conversely, if \(i\in J(n)\), then
\(\pi_i^\sigma(n)\in D_i^\sigma\) for each
\(\sigma\in\{0,1\}\). Hence
\(\bigl(\pi_i^\sigma(n),\pi_i^{1-\sigma}(n)\bigr)\) is an ordered
representation of \(n\) by elements of \(A\).

Suppose that \(i,k\in J(n)\) and
\(\bigl(\pi_i^\sigma(n),\pi_i^{1-\sigma}(n)\bigr)
=\bigl(\pi_k^\tau(n),\pi_k^{1-\tau}(n)\bigr)=(a,b)\).
If \(i\ne k\), then \(a\in D_i^\sigma\) and
\(b\in D_k^{1-\tau}\), so \(n\in E_{\mathrm{nc}}\), a contradiction.
Thus \(i=k\). If \(\sigma\ne\tau\), then
\(\pi_i^\sigma(n)=\pi_i^{1-\sigma}(n)\). Since
\(U_i^0\cap U_i^1=\varnothing\), both projections must be zero,
a contradiction to \(i\in J(n)\). Hence \(\sigma=\tau\). Therefore \(r_1(A,n)=2|J(n)|\).
\end{proof}

Let
$K(n)=\abs{\{i\ge1:T_i^2\le n\}}$,
and
$f(n)=2K(n)+1$.

\begin{proposition}\label{dens:prop:count}
For every \(n\notin E_0\), there exists \(b(n)\in\{0,1\}\) such that
\(
  r_1(A,n)=2K(n)+2b(n).
\)
Thus,
\(
  \abs{r_1(A,n)-f(n)}=1.
\)
\end{proposition}

\begin{proof}
Since \(n\notin E_{\mathrm{pr}}\), we have \(K(n)\le\abs{J(n)}\). Now suppose that \(i\in J(n)\) and \(T_i^2>n\). Since \(L_{i+1}=4L_i\), we have
\(
T_{i+1}=T_i^4,
\)
and therefore \(2T_{i+1}>T_i^2\). Hence the intervals
\([2T_i,T_i^2)\) are pairwise disjoint. Since every \(i\in J(n)\) with \(T_i^2>n\) satisfies
\(n\in[2T_i,T_i^2)\), there is at most one such index \(i\).
Therefore
\(
  \abs{J(n)}=K(n)+b(n)
\)
for some \(b(n)\in\{0,1\}\). By Lemma~\ref{dens:lem:rigidity}, we have
\(
  r_1(A,n)=2\abs{J(n)}
  =2K(n)+2b(n).
\)
Thus,
\(
  \abs{r_1(A,n)-f(n)}
  =\abs{2b(n)-1}
  =1.
\)
\end{proof}

\begin{proof}[Proof of Theorem~\ref{thm:density}]
For each \(i\ge1\) and \(\sigma\in\{0,1\}\), the set \(D_i^\sigma\)
is infinite. Hence \(A=\bigcup_{i\ge1}\bigl(D_i^0\cup D_i^1\bigr)\) is infinite. Let
\(
E_0=E_{\mathrm{nc}}\cup E_{\mathrm{pr}}.
\)
By Corollary~\ref{dens:cor:exceptional}, we have
\(|E_0\cap[1,x]|
=O\bigl(x^{15/16}(\log\log x)^2\bigr).
\)
By Proposition~\ref{dens:prop:count}, we have
\(
|r_1(A,n)-f(n)|=1
\)
for every \(n\notin E_0\). Therefore
\(
E:=\{n\in\mathbb N:|r_1(A,n)-f(n)|\ne1\}
\subseteq E_0,
\)
and hence
\[
|E\cap[1,x]|
=O\bigl(x^{15/16}(\log\log x)^2\bigr).
\]

It remains to verify the properties of \(f\). The function \(K\), and
hence \(f\), is increasing. For every \(r\ge1\), if \(n\ge T_r^2\),
then \(K(n)\ge r\). Thus \(K(n)\to\infty\) and \(f(n)\to\infty\).

Since
\(
T_i^2=4^{32\cdot4^i},
\)
the inequality \(T_i^2\le n\) implies
\(
i\le\log_4\left(\frac{\log_4 n}{32}\right).
\)
It follows that
\(
K(n)=O(\log\log n)\)
and
$f(n)=O(\log\log n)$.
In particular,
\[
f(n)=o\left(\frac{n}{(\log n)^2}\right).
\]

Finally, \(E\) has density zero, and for every \(n\notin E\),
\[
\frac{|r_1(A,n)-f(n)|}{\sqrt{f(n)}}
=\frac1{\sqrt{f(n)}}\longrightarrow0.
\]
Thus
\(
|r_1(A,n)-f(n)|=o\bigl(f(n)^{1/2}\bigr)
\)
on a set of integers of density \(1\).
\end{proof}

\end{document}